\documentclass[10pt,reqno]{amsart}
\usepackage{amsfonts,amsmath,layout}
\usepackage{mathrsfs}  
\usepackage{soul}
\usepackage{amssymb,mathabx,amsfonts,amsthm,amscd,stmaryrd,dsfont,esint,upgreek,constants,todonotes}
\DeclareFontFamily{OT1}{pzc}{}
\DeclareFontShape{OT1}{pzc}{m}{it}{<-> s * [1.10] pzcmi7t}{}
\DeclareMathAlphabet{\mathpzc}{OT1}{pzc}{m}{it}

\usepackage{color} % temporary

\definecolor{Red}{cmyk}{0,1,1,0.2}
\newcommand{\R}{\mathbb R}

\def\R{\mathbb R}

\def\ep{\epsilon}

\newcommand{\be}{\begin{equation}}
\newcommand{\ee}{\end{equation}}
\def\1{{\bf 1}}

\newtheorem{Theorem}{Theorem}[section]

\newtheorem{Lemma}[Theorem]{Lemma}

\begin{document}

%\title{}
%%\author{P. Cardaliaguet, N. Forcadel, R. Monneau} 
%\author{\renewcommand{\thefootnote}{\arabic{footnote}}
%  P. Cardaliaguet\footnotemark[1], N. Forcadel\footnotemark[2], R. Monneau\footnotemark[1] ~\footnotemark[3]}
%\footnotetext[1]{CEREMADE, UMR CNRS 7534, Universit\'e Paris Dauphine-PSL,
%Place de Lattre de Tassigny, 75775 Paris Cedex 16, France. }
%\footnotetext[2]{INSA Rouen Normandie, Normandie Univ, LMI UR 3226, F-76000 Rouen, France.}
%\footnotetext[3]{CERMICS, Universit\'e Paris-Est, Ecole des Ponts ParisTech, 6-8 avenue Blaise Pascal, 77455 Marne-la-Vall\'ee Cedex 2, France.}

\title[A note on contractive semi-groups on a 1:1 junction]{A note on contractive semi-groups on a 1:1 junction for  scalar conservation laws and Hamilton-Jacobi equations}
\author{P. Cardaliaguet}
\address{Universit\'e Paris-Dauphine, PSL Research University, Ceremade, 
Place du Mar\'echal de Lattre de Tassigny, 75775 Paris cedex 16 - France}
\email{cardaliaguet@ceremade.dauphine.fr }

\maketitle

\begin{abstract} We show that any continuous semi-group on $L^1$ which is (i) $L^1-$contractive, (ii) satisfies the conservation law $\partial_t \rho+\partial_x(H(x,\rho))=0$ in $\R_+\times (\R\backslash\{0\})$ (for a space discontinuous flux $H(x,p)= H^l(p) {\bf 1}_{x<0}+ H^r(p) {\bf 1}_{x>0}$), and  (iii) satisfies natural continuity and scaling properties, is necessarily given by a germ condition at the junction: $\rho(t,0)\in  \mathcal G$ a.e., where $\mathcal G$ is a maximal, $L^1-$dissipative and complete germ. In a symmetric way, we prove that any continuous semi-group on $L^\infty$ which is (i) $L^\infty-$contractive, (ii) satisfies with the Hamilton-Jacobi equation $\partial_t u+H(x,\partial_x u)=0$ in $\R_+\times (\R\backslash\{0\})$ (for a space discontinuous Hamiltonian $H$ as above), and  (iii) satisfies natural continuity and scaling properties, is necessarily given by a flux limited solution of the Hamilton-Jacobi equation. 
\end{abstract}

%\paragraph{AMS Classification:} 35L65, 76A30, 35B27, 35R02, 35D40, 35F20.
%\paragraph{Keywords:} traffic flow models, scalar conservation laws, homogenization on networks.

%\tableofcontents

\section{Introduction}

The aim of this note is to identify the conservation law (CL) semi-groups and the Hamilton-Jacobi (HJ) semi-groups which are $L^1-$  (for the CL) or $L^\infty-$  (for HJ) contracting, in the context of a 1:1 junction. Namely we consider a discontinuous flux (or Hamiltonian) $H:\R\times \R\to \R$ of the form
$$
H(x,p)= H^l(p){\bf 1}_{\{x<0\}}+ H^r(p){\bf 1}_{\{x>0\}}
$$
and look at the CL 
$$
\partial_t \rho+ \partial_x (H(x,\rho))= 0 \qquad \text{in}\; \R_+\times (\R\backslash \{0\})
$$
or the HJ equation
$$
\partial_t u+ H(x,\partial_x u)= 0 \qquad \text{in}\; \R_+\times (\R\backslash \{0\}).
$$
The problem is to find  conditions that have to be put at the junction $x=0$ to ensure that the equation is well-posed and $L^1-$contractive (for the CL) or $L^\infty-$contractive (for the HJ). Many papers have built such semi-groups for the CL: see, for instance,  \cite{AGDV11, AKR11, AuPe05, BaJe97, BKT09, GNPT07, To00, To01}. The main contribution of   \cite{AKR11} is to identify a general class of such semi-groups,  expressed through the notion of $\mathcal G-$entropy solutions: the additional condition takes the form  $(\rho(t,0^-), \rho(t,0^+))\in  \mathcal G$ for a.e. $t\geq0$, where the set $\mathcal G\subset \R^2$ is a maximal, $L^1-$dissipative and complete germ (see \cite{AKR11}). A somewhat symmetric construction is done in \cite{IM17} for the HJ equation, introducing the notion of flux limited solutions: the additional condition is now takes  the form:
$$
\partial_t u + \min \{\bar A, (H^l)^+(\partial_x^- u), (H^r)^-(\partial_x^+u)\}= 0 \qquad  \text{in} \; (0, \infty)\times \{0\}, 
$$
for a suitable flux limiter $\bar A$ (see below for the notations). The aim of this short note is to show that these semi-groups---described through germs (for CL) or flux limiters (for HJ)---are actually the only contractive ones. 

Our main results  (Theorem \ref{thm.main1} for the HJ equation and Theorem \ref{thm.main2} for the conservation law) state that, if a continuous semi-group is $L^1-$  or $L^\infty-$ contractive, coincides locally with the CL or HJ semi-group outside  $x=0$ and satisfies natural continuity and scaling properties, then it is necessarily given by a germ (for the CL) or a flux limiter (for the HJ). The choice to present at the same time the case of CL and of HJ is related to the fact that there is a strong relationship between both semi-groups: see \cite{CFGM} (for convex Hamiltonians) and  \cite{FIM24} (for coercive ones). We make a use of this property in the proofs. 

These results are somewhat expected: indeed, the introduction of \cite{AKR11} suggests that the notion of $\mathcal G-$entropy solutions was introduced to describe all the $L^1-$dissipative semi-groups for conservation law on an 1:1 junction. On the other hand, \cite{FIM24, IM17} show that, for the HJ equation, general conditions of the form $\partial_tu +F(\partial_x u)=0$  at the junction can be expressed in terms of flux limiter conditions. The main interest of this note is to actually {\it prove} the fact that natural contraction properties are enough to identify the semi-groups. To fix the ideas we work in the framework of traffic flow, i.e., assume that the fluxes/Hamiltonians $H^{l/r}$ are strictly concave. The results can probably be extended to more general fluxes or Hamiltonians, but to the expense of more technical proofs. 

Let us finally say a word about more complex junctions, i.e., junctions with more than one incoming branch or more than one outgoing one. The situation is then much more delicate: approaches by HJ and by CL are no longer equivalent (see for instance a counterexample in \cite{CFGM}), and the class of conditions that have to be put at the junction is much richer and---yet---far from understood.

%%%%%%%%%%%%%%%%%%%%%%%%%%%%%%%%%%%%%%%%
\section{Identification of the HJ equations}

We fix two smooth, Lipschitz and uniformly concave Hamiltonians $H^{l}: [0,R^l]\to \R$ and $H^r: [0,R^r]\to \R$, with $H^l(R^l)=H^r(R^{r})=H^l(0)=H^r(0)=0$ (for some $R^l, R^r>0$) and  set
$$
H(x,p)= H^l(p){\bf 1}_{\{x<0\}}+ H^r(p){\bf 1}_{\{x>0\}}.
$$
We define $Lip$ as the set of Lipschitz maps $u:\R\to \R$ with $u'\in [0,R^l]$ a.e. in $(-\infty, 0)$ and $u'\in [0,R^r]$ a.e. in $(0,\infty)$. The set $Lip$  is endowed with the topology of  local uniform convergence.  We denote by $\bar S^{HJ}$ the semi-group associated with the Hamilton-Jacobi equation on the line: given $u_0\in Lip$, $u(t,x):= \bar S^{HJ}(t,u_0)(x)$ is the {\it minimal} solution to 
$$
 \left\{\begin{array}{l}
\partial_t u +H(x,\partial_xu) = 0 \qquad \text{in}\; (0,\infty)\times \R\\
%\partial_xu\in [0,R] \qquad \text{in}\; (0,\infty)\times \R\\
u(0,x)= u_0(x)  \qquad \text{in}\; \R
\end{array}\right. 
$$
Then $u(t,\cdot)\in Lip$ for any $t\geq0$ and $\bar S^{HJ}$ is a continuous semi-group. Note that, if $H^l=H^r$, then $\bar S^{HJ}$ the standard semi-group associated with the Hamilton-Jacobi equation on the line. Here we need to specify that it is the minimal solution as the Hamiltonian is discontinuous.  

Our aim is to identify any semi-group $S^{HJ}$ satisfying the following conditions: 
\begin{align*}
(H1) & \qquad \text{$S^{HJ}:[0,\infty)\times Lip\to Lip$ is a continuous homogeneous semi-group,} \\
(H2) &\qquad   \text{$S^{HJ}$ is $L^\infty-$contractive: if $u_0^1, u_0^2\in Lip\cap L^\infty(\R), \;  \|S^{HJ}(t,u^0_1)-S^{HJ}(t,u_0^2)\|_\infty \leq \|u_0^1-u_0^2\|_\infty$,}\\
(H3) &\qquad \text{$S^{HJ}$ permutes with constants: $S^{HJ}(t, u_0+c)= S^{HJ}(t,u_0)+c$ for any  $c\in \R$ and  $u_0\in Lip$,} \\
(H4) &\qquad \text{$S^{HJ}$ has a finite propagation property: there exists $C_0>0$ such that,}\\
& \qquad \text{ if $u_0^1= u_0^2$ on $[a,b]$, then $S^{HJ}(t,u^0_1)= S^{HJ}(t,u_0^2)$ on $[a+C_0t, b-C_0t]$ for $t\in[0, C_0^{-1}(b-a)]$, }\\
(H5) &\qquad \text{$S^{HJ}$ coincides locally in $\R_+\times \R\backslash \{0\}$ with $\bar S^{HJ}$: for any $u_0\in Lip$ and $x\neq 0$, } \\
& \qquad \qquad \text{ $S^{HJ}(t,u_0)(y)=\bar S^{HJ}(t,u_0)(y)$ for any $t\in (0, C_0^{-1}|x|)$ and $y\in \R$ with $|y-x|<|x|-C_0t$,}\\ 
(H6) &\qquad \text{$S^{HJ}$ is scale invariant: for any $\ep,t >0$ and $u_0\in Lip$, $\ep S^{HJ}(t/\ep, \ep u_0(\cdot/\ep))(x/\ep)= S^{HJ}(t,u_0)(x)$.}
\end{align*}

\begin{Theorem}\label{thm.main1} Under the conditions above, there exists $\bar A\in [0,A_{\max}]$ such that $S^{HJ}=S^{HJ,\bar A}$, where $A_{\max}=\min\{\max H^l, \max H^r\}$ and  $S^{HJ,\bar A}$ is the HJ semi-group with flux-limiter $\bar A$, i.e., for a given initial condition $u_0\in Lip$, the viscosity solution to 
\be\label{eq.HJfluxlimitedCC}
 \left\{\begin{array}{l}
\partial_t u +H(x,\partial_xu) = 0 \qquad \text{in}\; (0,\infty)\times (\R\backslash\{0\})\\
\partial_t u + \min \{\bar A, (H^l)^+(\partial_x^- u), (H^r)^-(\partial_x^+u)\}= 0 \qquad  \text{in} \; (0, \infty)\times \{0\}\\
u(0,x)= u_0(x)  \qquad \text{in}\; \R
\end{array}\right. 
\ee
\end{Theorem} 

Above $(H^{l/r})^+$ (respectively $(H^{l/r})^-$) denotes the smallest nondecreasing (resp. nonincreasing) map above $H^{l/r}$. The notion of flux-limited viscosity solution of \eqref{eq.HJfluxlimitedCC} is defined as usual, except that the test functions $\phi=\phi(t,x)$, which are continuous in $\R_+\times \R$ and $C^1$ in $\R_+\times (\R\backslash \{0\})$, have only a left- and right- space derivative at $x=0$, denoted respectively by $\partial_x^- \phi(t,0)$ and $ \partial_x^+ \phi(t,0)$. This notion of viscosity flux limited solution for \eqref{eq.HJfluxlimitedCC} was introduced in \cite{IM17}, which also established that,  for any $A\in [0, A_{\max}]$, there exists a unique viscosity solution to \eqref{eq.HJfluxlimitedCC}. Note  that $\bar S^{HJ}= S^{HJ, A_{\max}}$. 
Recall finally that, by \cite{CT80}, under (H3), (H2) is equivalent to a comparison principle: $u_0^1\leq u_0^2\; \Rightarrow \; S^{HJ}(t,u^0_1)\leq S^{HJ}(t,u_0^2)$. 

\begin{proof} To simplify notation, we  set throughout the proof $S=S^{HJ}$, $\bar S=\bar S^{HJ}$ and $S^A=S^{HJ,A}$. Let us start with preliminary remarks. For $u_0\in Lip$, the map $(t,x)\to S(t,u_0)(x)$ is continuous and a solution of the HJ equation in $(0,\infty)\times (\R\backslash \{0\})$: by \cite[Theorem 2.7]{IM17} it is therefore a supersolution to the HJ equation on the whole space $(0,\infty)\times \R$. We infer by comparison that $S(t,u_0)\geq \bar S(t,u_0)$. 

Given $A\in [0, A_{\max}]$, let $p^{r/l,-}_A$ and $p^{r/l,+}_A$ be respectively the smallest and the largest solutions to $H^{r/l}(p)=A$. We define the maps 
$$
\check \phi_A(x)= p^{l,-}_A x {\bf 1}_{x\leq 0}+ p^{r,+}_A x {\bf 1}_{x\geq 0}, \qquad \hat \phi_A(x)= p^{l,+}_A x {\bf 1}_{x\leq 0}+ p^{r,-}_A x {\bf 1}_{x\geq 0} . 
$$
We  note that 
$$
S^A(t, \hat \phi_A) = \hat \phi_A- tA. 
$$
Indeed, the map $w(t,x):= \hat \phi_A(x)- tA$ is a classical solution to the HJ equation in $(0,\infty)\times (\R\backslash\{0\})$ and satisfies the initial condition. As for the condition at the vertex, it holds in a classical sense since 
$$
\partial_t w(t,0) + \min \{ A, H^{l,+}(\partial_x^- w(t,0)), H^{r,-}(\partial_x^+w(t,0))\}
= -A + \min \{ A, H^{l,+}(p^{l,+}_A), H^{r,-}(p^{r,-}_A)\} = 0.
$$
The idea of the proof is to first identify $S(t,\cdot)$ for the particular functions $\check \phi_A$ and $\hat \phi_A$, and then to use the $\check \phi_A$ and $\hat \phi_A$ as test functions. We do this through several steps: we first compute $S^A(t,\hat \phi_0)$ (Step 1), then define $\bar A$ and show, thanks to the first step, that $S(t,  \hat \phi_{0})= S^{\bar A}(t,\hat \phi_0)$ (Step 2): as $u\geq \hat \phi_0+ u(0)$ for any $u\in Lip$, this will give the supersolution property. To prove the subsolution property, we show that $S(t, \check \phi_A) = S^{\bar A}(t, \check \phi_A)$ (Step~3). We conclude in Step~4 by the proof of the equality $S=S^{\bar A}$. \\

\noindent {\it Step 1.} 
We first claim that, for any $A\in [0,A_{\max}]$, 
\be\label{ilkjzerdsCC}
S^A(t,\hat \phi_0)(x)= \left\{\begin{array}{ll}
\hat \phi_A(x)-tA & \text{if}\;  x\in [t(H^l)'(p^{l,+}_A), t(H^r)'(p^{r,-}_A)],\\
\bar S(t, \hat \phi_0)(x) & \text{otherwise}
\end{array}\right.
\ee
Moreover $u(t,x):=S^A(t,\hat \phi_0)(x)$ is the unique viscosity solution to the boundary value problem: 
\be\label{eq.boundarypbCC}
\left\{\begin{array}{l}
\partial_t u +H(x,\partial_xu) = 0 \qquad \text{in}\; (0,\infty)\times (\R\backslash\{0\}),\\
u(t,0) = - t A \qquad  \text{in} \; (0, \infty)\times \{0\},\\
u(0,x)= \hat \phi_0(x)  \qquad \text{in}\; \R.
\end{array}\right. 
\ee

\noindent {\it Proof of the claims:} Let us first compute $\bar S(t,\hat \phi_0)(x)$. We claim that 
$$
\bar S(t,\hat \phi_0)(x) = \left\{\begin{array}{ll}
0 & \text{if}  \; x\geq t (H^r)'(0),\\ 
- t (H^l_{A_{\max}})^*(x/t){\bf 1}_{x\leq 0} -t  (H^r_{A_{\max}})^*(x/t){\bf 1}_{x\geq 0} &  \text{if}   \;  x\in [ t (H^l)'(R^l), t (H^r)'(0)],\\ 
R^lx  &  \text{if}   \;  x\leq t (H^l)'(R^l),
\end{array}\right. 
$$
 where $(H^{l/p}_{A_{\max}})^*(p)= \sup_{y\in [0,R]} -py +\min\{H^{l/p}(y), A_{\max}\}$. Indeed, 
 let us  denote by $w(t,x)$ the function in the right-hand side. Then $w$ is continuous at $(t,0)$ because 
 $$
 w(t,0^\pm)= -t(H^{l/p}_{A_{\max}})^*(0) = -t\sup_{y\in [0,R]} \min\{H^{l/p}(y), A_{\max}\}= -tA_{\max}. 
 $$
 Moreover, $w$ is continuous at $(t, x)$ for $x:=t (H^r)'(0)>0$ since 
 $$
 w(t,x^-)=- t(H^r_{A_{\max}})^*(x/t)= - t\sup_{y\in [-R,0]} -(H^r)'(0)y+\min\{H^r(y), A_{\max}\}= tH^r(0)=0,
 $$ 
because the maximum in the second formula is reached at $y=0$ by the optimality condition. In the same way, for $x=t(H^l)'(R)<0$, 
$$
w(t,x^+)=-t(H^l_{A_{\max}})^*((H^l)'(R))= t(H^l)'(R^l)R^l= R^lx.
$$ 
In addition, $w$ is actually $C^1$ outside $x=0$, because $w$ is $C^1$ outside $x\in \{t (H^l)'(R^l), t (H^r)'(0),0\}$ and 
$$
\partial_x w(t, (t(H^r)'(0))^-)= -((H^r_{A_{\max}})^*)'((H^r)'(0))= 0,
$$ 
while $\partial_x w(t, (t(H^l)'(R^l))^+)=  -((H^l_{A_{\max}})^*)'((H^l)'(R^l))=R^l$. In addition, for $x=t (H^r)'(0)$, 
$$
\partial_t w(t,x^-)=  -(H^r_{A_{\max}})^*(x/t)+ (x/t) ((H^r_{A_{\max}})^*)'(x/t)= - (H^r_{A_{\max}})^*((H^r)'(0))+ (H^r)'(0) ((H^r_{A_{\max}})^*)'((H^r)'(0))= 0, 
$$
 and, for $x= t(H^l)'(R^l)$, 
 $$
\partial_t w(t,x^+)=  -(H^l_{A_{\max}})^*((H^l)'(R^l))+ (H^l)'(R^l) ((H^l_{A_{\max}})^*)'((H^l)'(R^l))= (H^l)'(R^l)R^l - (H^l)'(R^l) R^l= 0. 
$$
One also easily checks that $w$ satisfies the HJ equation in $(0,\infty)\times (\R\backslash\{0\})$ in a classical sense and the initial condition. Finally, at $x=0$, 
$$
\partial_t w(t,0)+ \min \{  A_{\max}, H^{l,+}(\partial_x^- w(t,0)), H^{r,-}(\partial_x^+w(t,0))\} = -A_{\max} + A_{\max}=0, 
$$
since $\partial_x^- w(t,0)= -((H^l_{A_{\max}})^*)'(0^-)=p^{l,-}_{A_{\max}}$ and $\partial_x^+w(t,0) = p^{r,+}_{A_{\max}}$. 
So $w(t,x)=\bar S(t,\hat \phi_0)(x)$.\\

Next we denote by $u$ the function in the right-hand side of \eqref{ilkjzerdsCC}. We note that, if $A=A_{\max}$, the result is obvious because $\bar S= S^{A_{\max}}$ and we have seen that $ w(t,0^\pm)=-tA_{\max}$. Let us now assume that $A\in[0, A_{\max})$. Let us check that $u$ is continuous at the points $(t, t(H^l)'(p^{l,+}_A))$ and   $(t,t(H^r)'(p^{r,-}_A))$. Indeed, for $x=   t(H^l)'(p^{l,+}_A)\in [ t (H^l)'(R^l), t (H^r)'(0)]$, 
$$
\hat \phi_A(x)-tA = p^{l,+}_A t(H^l)'(p^{l,+}_A)- t \min\{H^l(^{l,+}_A), A_{\max}\}= -t (H^l_{A_{\max}})^*((H^l)'(^{l,+}_A))= w(t,x). 
$$
A symmetric  argument shows that $u$ is continuous at $(t,t(H^r)'(p^{l,-}_A))$.
We now claim that $u$ is $C^1$ near the points $(t, t(H^l)'(p^{l,+}_A))$ and   $(t,t(H^r)'(p^{l,-}_A))$. As the maps $(t,x)\to \hat \phi_A(x)-tA$ and $w$ are $C^1$ near these points, we just need to check that the derivatives coincide at these points. Indeed, for $x=   t(H^l)'(p^{l,+}_A))$,  
$$
\partial_x w(t,  x)= ((H^l_{A_{\max}})^*)'((H^l)'(p^{l,+}_A))= p^{l,+}_A = \partial_x \hat \phi_A(t(H^l)'(p^{l,+}_A)), 
$$
while
\begin{align*}
\partial_t w(t, x) & = -(H^l_{A_{\max}})^*(((H^l)'(p^{l,+}_A))+ (H^l)'(p^{l,+}_A) ((H^l_{A_{\max}})^*)'((H^l)'(p^{l,+}_A))\\
& = (H^l)'(p^{l,+}_A)p^{l,+}_A -H^l(p^{l,+}_A)-(H^l)'(p^{l,+}_A)p^{l,+}_A= -A .
\end{align*}
A symmetric  argument shows that $u$ is $C^1$ near $(t,t(H^r)'(p^{l,-}_A))$.
This also implies that $u$ is a solution to the HJ equation in $(0,\infty)\times (\R\backslash\{0\})$ and satisfies the initial condition. As $u(t,0)= -tA$, $u$ solves therefore the boundary value problem \eqref{eq.boundarypbCC}. Finally, $u$ coincides near $x=0$ with $\hat \phi_A(x)-tA$, which is a solution to the HJ equation with flux limiter $A$. Thus $u$ is also a solution to the HJ equation with flux limiter $A$: this proves the claim.
\bigskip

\noindent {\it Step 2.}   Let us set $\bar A= - S(1, \hat \phi_0)(0)$. We now claim that $\bar A\in [0, A_{\max}]$ and that
\be\label{lkaj:zernCC}
S(t,  \hat \phi_{0})= S^{\bar A}(t,\hat \phi_0).
\ee
while 
\be\label{lkaj:zern2CC}
S(t, \hat \phi_{\bar A})=\hat \phi_{\bar A}- t\bar A= S^{\bar A}(t, \hat \phi_{\bar A}).
\ee

{\it Proof of the claim: } Let $u(t,x):= S(t,  \hat \phi_{0})(x)$. By the HJ equation, $\partial_t u \in [- A_{\max}, 0]$ a.e.. Hence $\bar A\in [0,A_{\max}]$. 
We now prove \eqref{lkaj:zernCC}. By assumption (H6), the map $w(t,x):= S(t, \hat \phi_{\bar A})(x)$ is scale invariant: for any $\ep>0$, $\ep w(t/\ep, x/\ep)= w(t,x)$. Thus $w(t,x)= t w(1, x/t)$. As a consequence, $w(t,0)= -t \bar A$ for any $t\geq 0$ and thus $w$ is the unique viscosity solution to the boundary value problem \eqref{eq.boundarypbCC} with $A:=\bar A$. By Step~1, we obtain $w(t,x)= S^{\bar A}(t,\hat \phi_0)(x)$, which is \eqref{lkaj:zernCC}. 

To prove  \eqref{lkaj:zern2CC}, we  use again the self similarity to obtain that $S(t, \hat \phi_{\bar A})(x)= tS(1, \hat \phi_{\bar A})(x/t)$. Now, by  \eqref{ilkjzerdsCC} and for any $t>0$, there is some $\delta>0$ such that $u(s,x)= S^{\bar A}(s,\hat \phi_0)(x)$ coincides with $\hat \phi_{\bar A}(x)-s\bar A$ in $[t,t+\delta]\times [-\delta, \delta]$. Thus, by finite speed of propagation (H4) and semi-group properties, 
$$
S(h, \hat \phi_{\bar A}(\cdot)-t\bar A)(x)=S(h,u(t,\cdot))(x)= u(t+h,x)= \hat \phi_{\bar A}(x)-(t+h)\bar A
$$
for $h>0$ and $|x|$ small enough. This implies (using (H3)) that $S(h, \hat \phi_{\bar A}(\cdot))(0)= -h\bar A$, and thus that $S(1, \hat \phi_{\bar A})(0)= -\bar A$. We infer that $(t,x)\to S(t, \hat \phi_{\bar A})(x)$ is a solution to the boundary value problem \eqref{eq.boundarypbCC} with $A:=\bar A$ and initial condition $\hat \phi_{\bar A}$.   But $(t,x)\to \hat \phi_{\bar A}(x)-t\bar A$ is  the unique solution to this boundary value problem, which proves \eqref{lkaj:zern2CC}.
\bigskip

\noindent {\it Step 3.}  We next claim that, for any $A\in [0,\bar A]$, 
$$
S(t, \check \phi_A) (0)=-At. 
$$

{\it Proof of the claim:} By scale invariance, we know that $S(t, \check \phi_A)(x)= t S(1, \check \phi_A)(x/t)$. Recall also that $S(t,  \check \phi_A)\geq \bar S(t , \check \phi_A)=  \check \phi_A- tA$. Thus $S(1, \check \phi_A)(0)\geq -A$. Our aim is to prove the equality. Let us assume for a while that 
$\beta:= -S(1, \check \phi_A)(0) <  A$. Then $(t,x) \to S(t, \check \phi_A)(x)$ is the unique viscosity solution to the boundary value problem 
\be\label{eq.BVPbetaCC}
\left\{\begin{array}{l}
\partial_t u +H(x,\partial_xu) = 0 \qquad \text{in}\; (0,\infty)\times (\R\backslash\{0\})\\
u(t,0) = - t \beta \qquad  \text{in} \; (0, \infty)\\
u(0,x)= \check \phi_A(x)  \qquad \text{in}\; \R
\end{array}\right. 
\ee
We note that 
\be\label{iuylzekuhrjdf}
S(t, \check \phi_A)(x) =\max\{  \check \phi_A(x)-At , \hat \phi_\beta(x)-\beta t\}.
\ee
Indeed, the function on the right-hand side satisfies the boundary conditions at $x=0$ (since $\beta<A$) and at $t=0$ (since $\check \phi_A\geq \hat \phi_\beta$), and is a viscosity solution in $\R_+\times (\R\backslash \{0\})$ as the maximum of two viscosity solutions, the Hamiltonian being concave. Thus \eqref{iuylzekuhrjdf} holds.  

Given $t>0$, there exists $\delta>0$ such that $S(s, \check \phi_A)(x) =  \hat \phi_\beta(x)-\beta s$ on $[t, t+\delta]\times [-\delta,\delta]$. Hence, by localization,  
$$
S(h,  \hat \phi_\beta)(x)= S(t+h, \check \phi_A)(x) +\beta (t+h)= \hat \phi_\beta(x)-\beta h
$$
for $h>0$ small and $|x|$ small. Thus $S(h,  \hat \phi_\beta)(0)= -\beta h$ for $h>0$ small. Using the scaling invariance, we infer that $S(1, \hat \phi_\beta)(0)=-\beta$. On the other hand, as  $\hat \phi_\beta \leq \hat \phi_{\bar A}$ (because $\beta<A\leq \bar A$), we have by comparison and \eqref{lkaj:zern2CC}: 
$$
-\beta = S(1, \hat \phi_\beta)(0)\leq S(1, \hat \phi_{\bar A})(0)= -\bar A, 
$$
a contradiction with $\beta < \bar A$. So we have proved the equality $S(1, \check \phi_A)= -A$. 
\bigskip

\noindent {\it Step 4.}  We finally prove that $S= S^{\bar A}$. Fix $u_0\in Lip$ and set $u(t,x)=S(t,u_0)(x)$. Let us check that $u$ solves \eqref{eq.HJfluxlimitedCC}. Recall that $u$ solves the HJ equation in $\R_+\times (\R\backslash\{0\})$ and the initial condition $u_0$ at $t=0$. It remains to check the viscosity solution property at $x=0$. 

We first prove that $u$ is a supersolution to \eqref{eq.HJfluxlimitedCC}. As $u(t, \cdot)\in Lip$, we have, for any $t\geq 0$,
$$
u(t,x) \geq u(t,0) + \hat \phi_0(x).
$$
Hence, by comparison (H2),  commutation with constants (H3) and semi-group property (H1), 
$$
u(t+h,x) \geq u(t,0) + S(h, \hat \phi_0)(x)\qquad \forall h>0, \; \forall x\in \R.
$$
So, by \eqref{lkaj:zernCC} and \eqref{ilkjzerdsCC} (applied to $x=0\in [t(H^l)'(p^{l,+}_A), t(H^r)'(p^{r,-}_A)]$),  
$$
u(t+h, 0) \geq u(t,0) -h\bar A,  
$$ 
which shows that $\partial_t u(t,0) +\bar A \geq 0$ in the viscosity sense. By \cite[Theorem 2.7]{IM17} (written in \cite{IM17} for convex Hamiltonians while here $H$ is concave), $u$ is therefore a viscosity supersolution to \eqref{eq.HJfluxlimitedCC}. 

We now prove that $u$ is a subsolution to \eqref{eq.HJfluxlimitedCC}. Let $\varphi:\R\to \R$ be a $C^1$ test function on the junction (meaning that $\varphi$ is continuous in $\R$, $C^1$ in $(-\infty, 0]$ and in  $[0,\infty)$) and $\psi$ be a $C^1$ test function such that $(t,x)\to u-\varphi(x)-\psi(t)$ has a local maximum at some $(\bar t, 0)$ with $\bar t>0$. We have to check that 
\be\label{cond.sursolCC}
\partial_t\psi (\bar t) + \min\{ \bar A, H^{l,+}(\partial_x^-\varphi(0)),H^{r,-}(\partial_x^+\varphi(0))\} \leq 0. 
\ee
Without loss of generality, we suppose that $\varphi(0)=0$ and $\psi(\bar t)= u(\bar t, 0)$. Following \cite[Theorem 2.7]{IM17}, we can also assume that \be\label{mozesrdijft}
H^l(\partial_x^-\varphi(0))=H^r(\partial_x^+\varphi(0))=\bar A. 
\ee
Note that \eqref{cond.sursolCC} is obvious if $\bar A=0$ because $u$ is nonincreasing in time and thus $\psi'(\bar t)\leq 0$. 
We now suppose that  $\bar A>0$. Fix $A\in [0,\bar A)$. Then, because of  \eqref{mozesrdijft}, we have locally around $x=0$, $\varphi(x) \leq \check \phi_A(x)$ with an equality at $x=0$. Hence $u(t,x) \leq \check \phi_A(x)+\psi(t)$ around $(\bar t, 0)$. By (H2), (H3) and (H4), this implies that, for $h>0$  and $|x|$ small,  
$$
u(\bar t,x) = S(h, u(\bar t-h, \cdot))(x) \leq S(h, \check \phi_A)(x) +\psi(\bar t-h). 
$$
Using Step 3, we infer that 
$$
\psi(\bar t)= u(\bar t, 0) \leq -A h +\psi(\bar t-h). 
$$
Hence 
$$
\psi'(\bar t) +A \leq 0. 
$$
As $A<\bar A$ is arbitrary, we can conclude that $\psi'(\bar t) +\bar A \leq 0$ and thus that \eqref{cond.sursolCC} holds.  

In conclusion, $u=S(u_0)$ is a viscosity solution to \eqref{eq.HJfluxlimitedCC}, which proves that $S=S^{\bar A}$. 
\end{proof} 

%%%%%%%%%%%%%%%%%%%%%%%%%%%%%%%%%%%%%%%
\section{Identification of the CL}

As in the previous part, we fix two smooth, Lipschitz and uniformly concave Hamiltonians  $H^{l}: [0,R^l]\to \R$ and $H^r: [0,R^r]\to \R$, with $H^l(R^l)=H^r(R^{r})=H^l(0)=H^r(0)=0$ (for some $R^l, R^r>0$) and  set
$$
H(x,p)= H^l(p){\bf 1}_{\{x<0\}}+ H^r(p){\bf 1}_{\{x>0\}}.
$$
We now denote by $L^\infty$ the set of measurable maps $\rho$ such that $\rho(x)\in [0,R^l]$ a.e. in $(-\infty, 0)$ and $\rho(x)\in [0,R^r]$ a.e. in $(0,\infty)$. We endow the set $L^\infty$ with the topology of the $L^1_{loc}$ convergence. We denote by $\bar S^{CL,l}$ (resp. $\bar S^{CL,r}$)  the classical semi-group of conservation law in $(0,\infty)\times \R$ associated with the flux $H^{l}$ (resp. $H^{r}$).

Our aim is to identify any semi-group $S^{CL}$ on $L^\infty$ satisfying the following properties: 
\begin{align*}
(H1) & \qquad \text{$S^{CL}:[0,\infty)\times L^\infty\to L^\infty$ is a continuous homogeneous  semi-group,} \\
(H2) &\qquad   \text{$S^{CL}$ satisfies the $L^1$ contraction property: for any $u^1_0,u^2_0\in L^\infty\cap L^1(\R)$, }\\
& \qquad \qquad \text{$\|S^{CL}(t,u^0_1)- S^{CL}(t,u_0^2)\|_{L^1(\R)}\leq \|u^1_0-u^2_0\|_{L^1(\R)}$,}  \\
(H3) &\qquad \text{$S^{CL}$ preserves mass: if $\rho_0\in L^\infty\cap L^1(\R)$, then $\int_\R S^{CL}(t, u_0)= \int_\R u_0$ for any $t\geq 0$,} \\
(H4) &\qquad \text{$S^{CL}$ has a finite speed of propagation: there exists $C_0>0$ such that,}\\
& \qquad \qquad\text{ if $u_0^1= u_0^2$ on $[a,b]$, then $S^{CL}(t,u^0_1)= S^{CL}(t,u_0^2)$ on $[a+C_0t, b-C_0t]$ for $t\in[0, C_0^{-1}(b-a)]$, }\\
(H5) &\qquad \text{$S^{CL}$ coincides locally with $\bar S^{CL,l/r}$ away from $x=0$: for any $u_0\in L^\infty$ and $x\neq 0$, } \\
&  \qquad \qquad \text{$S^{CL}(t,u_0)(y)=\left\{ \begin{array}{ll}
\bar S^{CL,l}(t,u_0)(y)& \text{if}\; x<0\\
\bar S^{CL,r}(t,u_0)(y)& \text{if}\; x>0
\end{array}\right.$  for any $(t,y)$ with $|y-x|<|x|-C_0t$,}\\ 
(H6) &\qquad \text{$S^{CL}$ is scale invariant: for any $\ep,t >0$ and $u_0\in L^\infty$, $ S^{CL}(t/\ep,  u_0(\cdot/\ep))(x/\ep)= S^{CL}(t,u_0)(x)$.}
\end{align*}

\begin{Theorem}\label{thm.main2}  Under the conditions above, there exists $\bar A\in [0,A_{\max}]$ such that $S^{CL}=S^{CL,\bar A}$, where $A_{\max}=\min\{\max H^l,\max H^r\}$ and  $S^{CL,\bar A}$ is the conservation law semi-group with flux-limiter $\bar A$, i.e., for a given initial condition $\rho_0\in L^\infty$, the entropy solution to 
\be\label{eq.CLgerm}
 \left\{\begin{array}{l}
\partial_t \rho +\partial_x(H(x,\rho)) = 0 \qquad \text{in}\; (0,\infty)\times (\R\backslash\{0\})\\
(\rho(t,0^-),\rho(t,0^+))\in \mathcal G_{\bar A} \qquad  \text{a.e. in} \; (0, \infty)\\
\rho(0,x)= \rho_0(x)  \qquad \text{in}\; \R
\end{array}\right. 
\ee
where the germ $\mathcal G_{\bar A}$ is given by 
$$
\mathcal G_{\bar A}=\left\{(q^-,q^+)\in [0,R^l]\times [0,R^r], \; H^l(q^-)=H^r(q^+)= \min \{\bar A, H^{l,+}(q^-),H^{r,-}(q^+)\} \; \right\}.
$$
\end{Theorem}

Above $(H^{l/r})^+$ (respectively $(H^{l/r})^-$) denotes the smallest nondecreasing (resp. nonincreasing) map above $H^{l/r}$. 
The notion of $\mathcal G-$entropy solution for \eqref{eq.CLgerm} is introduced in \cite{AKR11}: it is a Kruzhkov entropy solution in $(0,\infty)\times (\R\backslash\{0\})$, its trace at $t=0$ is $\rho_0$ and its trace at $x=0$ belongs to  $\mathcal G_{\bar A}$ (the existence of a strong trace being given by \cite{Ps07, V01}). As $\mathcal G_{\bar A}$ is a maximal, $L^1-$dissipative and complete germ (see \cite{CFGM}---see also  \cite{AKR11} for the definition of these notions, which are not explicitly used here), the solution to \eqref{eq.CLgerm}  exists and is unique (see \cite{AKR11}). \\

We recall that, because of Assumption (H3), Assumption (H2) is equivalent to a comparison principle for $S^{CL}$: if $\rho_{0,1}, \rho_{0,2}\in L^\infty$ and  $\rho_{0,1}\leq \rho_{0,2}$, then $S^{CL}(\rho_{0,1})\leq S^{CL}(\rho_{0,2})$ (see \cite{CT80}). \\ 

Throughout the proof, we set $S= S^{CL}$, $S^A= S^{CL, A}$ and $\bar S^{l/r}= \bar S^{CL,l/r}$ for simplicity. We also assume without loss of generality that $C_0\geq \|H'\|_\infty$, so that the CL semi-groups $\bar S^{l/r}$ also have a finite speed of propagation  $C_0$.  \\ 

Our first step is a standard consequence of the mass preservation property: 

\begin{Lemma}\label{lem.RH} Let $\rho_0\in L^\infty$ and $\rho= S(\rho_0)$. Then $u$ satisfies a kind of Rankine-Hugoniot condition at $x=0$:
$$
H^l(\rho(t, 0^-))= H^r(\rho(t, 0^+))\qquad \text{ for a.e. $t\geq 0$.}
$$
\end{Lemma}

\begin{proof} As the property we aim to prove is local near $x=0$ and $S$ satisfies a finite speed of propagation, we can assume that $\rho_0$ has a compact support. Let $\psi:\R\to [0,1]$ be a smooth nonnegative test function equal to $1$ on $\R\backslash(-2,2)$ and vanishing on $[-1,1]$. 
For $\ep>0$, we set $\psi_\ep(x)= \psi(x/\ep)$. As $\rho$ is a solution of the CL in $\R_+\times (\R\backslash \{0\})$, we have, for any $t\geq0$, 
$$
\int_\R \rho(t)\psi_\ep + \int_0^t \int_\R H(x,\rho) \psi_{\ep, x} = \int_\R \rho_0\psi_\ep
$$
We let $\ep\to 0$ to find 
$$
\int_\R \rho(t) + \int_0^t  (H^r(\rho(s,0^+))-H^l(\rho(s,0^-)))ds = \int_\R \rho_0,
$$
which proves the claim thanks to (H3). 
\end{proof}

The next remark makes the link between the CL and the HJ equation. 

\begin{Lemma}\label{lem.HJvsCL} Let $\rho_0\in L^\infty$ have a compact support and assume that $\rho$ is an entropy solution of the CL in $\R_+\times (\R\backslash\{0\})$ with initial condition $\rho_0$. Let  $u(t,x):= \int_{-\infty}^x \rho(t,y)dy$. Then $u$ is a viscosity solution to the HJ equation 
\be\label{eq.HJbis}
\left\{\begin{array}{l}
\partial_t u +H(x,\partial_xu) = 0 \qquad \text{in}\; (0,\infty)\times (\R\backslash\{0\})\\
%u(t,0) = - t A +u_0(0)\qquad  \text{in} \; (0, \infty)\\
u(0,x)= u_0(x)  \qquad \text{in}\; \R
\end{array}\right. 
\ee
where $u_0(x):= \int_{-\infty}^0 \rho_0(y)dy$. If, in addition, there exists $(q^-,q^+) \in [0,R^l]\times [0,R^r]$  such that  
\be\label{hyp.ctbd}
(\rho(t, 0^-), \rho(t,0^+))= (q^-,q^+)\qquad \text{for a.e. $t\geq 0$,}
\ee
then $u$ satisfies also the boundary condition 
\be\label{eq.boundarycond}
u(t,0) = - t A +u_0(0)\qquad  \text{in} \; (0, \infty)
\ee
where $A:= H^l(q^-)=H^r(q^+)$. 

Finally, for any $\rho_0\in L^\infty$,  if $\rho:=S(\rho_0)$ satisfies \eqref{hyp.ctbd}, there exists $u_0\in Lip$ such that $\rho=\partial_x u$ where $u$  is the unique viscosity solution to the boundary value problem \eqref{eq.HJbis}-\eqref{eq.boundarycond}. 
\end{Lemma}

\begin{proof} The fact that $u$ is a viscosity solution in $(0,\infty)\times (\R\backslash\{0\})$ comes from assumption (H5): Indeed, fix $t>0$ and $x\neq 0$. Let $w(s,\cdot)=\bar S^{HJ}(s, u( t, \cdot))$. Then it is known (see, for instance, \cite{CFM24} for a local version of the statement) that $\partial_x w$ is an entropy solution of the CL with initial condition $\partial_x u( t, \cdot)= \rho( t, \cdot)$. As, by (H5), $\rho$ is also an entropy solution in the triangle 
$$
\Delta:= \{(s,y),\; s\in (t, t-C_0^{-1}|x|), \; y\in (x-|x|+C_0(s- t), x+|x|-C_0(s- t))\},
$$
we infer by the uniqueness of the entropy solution that   $\partial_x w(s,\cdot)= \rho(s, \cdot)= \partial_x u(s,\cdot)$ in $\Delta$. On the other hand, for any test function $\varphi$ with a compact support in $\Delta$, we have, by the definition of $u$ and since $w$ is a solution of the HJ equation while $\rho$ solves the CL,  
$$
\iint_{\Delta} u  \varphi_t = \iint_{\Delta} \rho  \tilde \varphi_t =   - \iint_{\Delta} H(x,\rho) \tilde \varphi_x = \iint_{\Delta} H(x,\rho)  \varphi =
  \iint_{\Delta} H(x,\partial_x w) \varphi = - \iint_{\Delta} w_t \varphi. 
$$
where we have set $\tilde \varphi(s,y)= \int_y^\infty \varphi(s,x)dx.$
This proves that $u_t= w_t$ in $\Delta$, and thus that $u=w$ in $\Delta$ since $u(t,\cdot)= w(t, \cdot)$. Hence $u$ is a viscosity solution to the HJ equation away from $x=0$. In addition, by the continuity of $t\to \rho(t)$ in $L^1_{loc}$, $u$ is continuous up to time $t=0$ and $u(0,x)= u_0(x)$. 

Let us now check the boundary condition at $x=0$ under the addition assumption \eqref{hyp.ctbd}. Let $\psi:\R\to [0,1]$ be a smooth nonnegative test function equal to $1$ on $\R\backslash(-2,2)$ and vanishing on $[-1,1]$.  For $\ep>0$, we set $\psi_\ep(x)= \psi(x/\ep)$. As $\rho$ is a weak solution to the CL in $\R_+\times (-\infty,0)$, we have 
$$
\int_{-\infty}^0 \rho(t)\psi_\ep - \int_0^t\int_{-\infty}^0 H(x,\rho)(\psi_\ep)_x = \int_{-\infty}^0 \rho(0)\psi_\ep
$$
We let $\ep\to 0$. Then $\int_{-\infty}^0 \rho(t)\psi_\ep$ tends to $u(t,0)$, $\int_{-\infty}^0 \rho(0)\psi_\ep$ tends to $u_0(0)$ while $\int_0^t\int_{-\infty}^0 H(x,\rho)(\psi_\ep)_x$ tends to $-\int_0^t H(x,\rho(s, 0^-))ds$. Assumption \eqref{hyp.ctbd} then implies the boundary condition $u(t,0) = - t A +u_0(0)$ for $t\geq0$. 

For the last part of the proof, we  proceed by approximation: given $M>0$, let $\rho_{0,M}= \rho_0{\bf 1}_{[-M,M]}$, $\rho_M= S(\rho_{0,M})$ and let $u_M$ be built as above. Note that $\rho_M$ tends to $\rho$ in $L^1_{loc}$ as $M\to \infty$ by the continuity of the semi-group. Then, for any $T>0$ and  $M$ large enough, $\rho_M$ satisfies \eqref{hyp.ctbd} on $[0,T]$, so that $\tilde u_M:=u_M-u_{M}(0,0)$ solves \eqref{eq.HJbis}-\eqref{eq.boundarycond} on the time interval $[0,T]$ and satisfies $\rho_M=\partial_x \tilde u_M$. By compactness, $(\tilde u_M)$ converges,  as $M\to \infty$ and up to a subsequence, to some solution $u$ with the required properties on the time interval $[0,\infty)$ since $T$ was arbitrary. 
\end{proof}

The next remark is about blow-up limits of solutions at points $(t,0)$. It is a central consequence of the existence of strong traces for  solutions of CL  in the sense of Vasseur \cite{V01} and Panov \cite{Ps07} and of strong trace for the gradient of solution of HJ equation in the sense of Monneau \cite{Mon23}. 

\begin{Lemma}\label{lem.statsol} Let $\rho_0\in L^\infty$ and, for $\bar t>0$ and $\ep>0$, set $\rho^\ep_{\bar t}(t,x)= S(\bar t+\ep t,\rho_0)(\ep x)$. Then for a.e. $\bar t>0$ and as $\ep\to 0^+$, the sequence $\rho_{\bar t}^\ep$ converges in $L^1_{loc}$  to a constant (in time) function  $\bar \rho_{\bar t}(x)=q^-_{\bar t}{\bf 1}_{\{x<0\}}+q^+_{\bar t}{\bf 1}_{\{x>0\}}$, where $(q^-_{\bar t},q^+_{\bar t})= (S(\bar t,\rho_0)(0^-),S(\bar t,\rho_0)(0^+))$, which is a stationary solution for $S$: $S(t, \bar \rho_{\bar t})= \bar \rho_{\bar t}$ for any $t\geq 0$. 
\end{Lemma}

\begin{proof} As the result holds in a neighborhood of $x=0$, we can assume without loss of generality that $\rho_0$ has a compact support. Let us set $u(t,x):= \int_{-\infty}^x \rho(t,y)dy$. We know from Lemma \ref{lem.HJvsCL} that $u$ is a viscosity solution to the HJ equation in $\R_+\times (\R\backslash \{0\})$. Let $\bar t$ be such that the Lipschitz function $u(t,0)$ has a derivative at $\bar t$. Following  \cite{Mon23}, $u$ then admits a full half-derivative at $(\bar t,0^\pm)$: namely, there exists $q^\pm$ and $q_t\in \R$ such that
$$
u^\ep(t,x)=\ep^{-1}\left( u(\bar t+\ep t, \ep x)- u(\bar t,0)\right) \to u^0(t,x):=tq_t + x q^-{\bf 1}_{x<0}+ x q^+{\bf 1}_{x>0}\qquad \text{as $\ep\to 0^+$,}
$$
the convergence being locally uniform in $\R\times (-\infty,0)$ and in $\R\times (0,\infty)$. Let now $\rho^\ep_{\bar t}(t,x)= S(\bar t+\ep t,\rho_0)(\ep x)$. By the semi-group and the scaling properties, $\rho^\ep$ is a solution for $S$: $S(t, \rho^\ep_{\bar t}(0, \cdot))= \rho^\ep_{\bar t}(t,\cdot)$. As $\rho^\ep$ is a bounded solution to the CL with a uniformly concave flux $H$ in $\R_+\times (\R\backslash\{0\})$, the family $(\rho^\ep)$ is bounded in $BV_{loc}(\R_+\times (\R\backslash\{0\}))$ and therefore relatively compact in $L^1_{loc}(\R_+\times \R)$. Let $(\ep_n)$ be such that $(\rho^{\ep_n})$ converges to some $\bar \rho$ in $L^1_{loc}$. We claim that 
\be\label{lqeisundf}
\bar \rho(t,x)= q^-{\bf 1}_{x<0}+ q^+{\bf 1}_{x>0}.
\ee
Note that the claim implies the convergence of the whole sequence $(\rho^\ep)$ to $\bar \rho$. To prove the claim we note that  $\partial_x u^\ep= \rho^\ep$. As $u^\ep$ converges to $u^0$, we infer that $\rho^\ep=\partial_x u^\ep$ converges to $\partial_x u^0= q^-{\bf 1}_{x<0}+ q^+{\bf 1}_{x>0}$ in $L^\infty-$weak-*. But  $\rho^\ep$ converges to $\bar \rho$ in $L^1_{loc}$, so that \eqref{lqeisundf} holds. As $\rho^\ep$ is a solution for $S$, $\bar \rho$ is also a solution for $S$. From now on we set $(q^-_{\bar t}, q^+_{\bar t})= (q^-,q^+)$ to emphasize the dependence in $\bar t$. Recall that $(q^-_{\bar t}, q^+_{\bar t})$ exists for a.e. $t\geq 0$. 

Let us finally check that  $(q^-_{\bar t},q^+_{\bar t})= (S(\bar t,\rho_0)(0^-),S(\bar t,\rho_0)(0^+))$ for a.e. $\bar t$. Set $\rho:= S(\cdot, \rho_0)$. By the strong trace property (see \cite{Ps07,V01}), for any $b>0$,  
$$
\text{ess-lim}_{x\to 0^+} \int_0^b |\rho(t,x)-\rho(t,0^+)|dt = 0. 
$$
Hence (by a simple change of variable)
$$
\lim_{\ep \to 0} \int_0^1\int_0^b \int_0^1  |\rho(t+\ep s,\ep x)-\rho(t+\ep s,0^+)|dx dt ds = 0. 
$$
By the $L^1$ convergence of $\rho^\ep_{\bar t}$ for a.e. $\bar t$ to $q^+_{\bar t}$ in $\R\times (0,\infty)$ and the $L^1$ convergence of $\rho(\cdot+\ep s,0^+)$ to $\rho(\cdot,0^+)$, we infer that 
$$
\int_0^1\int_0^b \int_0^1  |q^+_{t}-\rho(t,0^+)|dx dt ds = 0, 
$$
from which we conclude that $q^+_{t}= \rho(t,0^+)$ for a.e. $ t$. The proof of the a.e. equality $q^+_{ t}= S(\bar t,\rho_0)(0^+)$ is symmetric. 
\end{proof}

Let us now introduce a few notation. Given $A\in [0, A_{\max}]$, let $p^{r/l,-}_A$ and $p^{r/l,+}_A$ be respectively the smallest and the largest solutions to $H^{r/l}(p)=A$.  We define the maps 
\be\label{def.psipsipsi}
\check \psi_A(x)= p^{l,-}_A {\bf 1}_{x\leq 0}+ p^{r,+}_A  {\bf 1}_{x\geq 0}, \qquad \hat \psi_A(x)= p^{l,+}_A  {\bf 1}_{x\leq 0}+ p^{r,-}_A  {\bf 1}_{x\geq 0} . 
\ee

\begin{Lemma}\label{lem.psicannot} Assume that there exists $\bar A\in [0, A_{\max}]$ such that $\hat \psi_{\bar A}$ is a stationary solution for $S$: $S(t, \hat \psi_{\bar A})= \hat \psi_{\bar A}$ for any $t\geq 0$. Then, for any $\alpha\in [0,A_{\max}]\backslash\{\bar A\}$ and any $\beta \in (\bar A, A_{\max}]$, $\hat \psi_\alpha$ and  $\check \psi_\beta$ are not  stationary solutions for $S$. 
 \end{Lemma}

\begin{proof} We start with a preliminary remark. Fix $\alpha, \beta \in [0, A_{\max}]$ with $\alpha\neq \beta$. We claim that $\hat \psi_\alpha$ and $\hat \psi_\beta$ cannot be two stationary solutions for $S$. To fix the ideas, we suppose that $\alpha<\beta$ and argue by contradiction, assuming that $\hat \psi_\alpha$ and $\hat \psi_\beta$ are two stationary solutions for $S$.

For $A\in [0,A_{\max}]$, let us set $\hat \phi_A(x)= p^{l,+}_Ax{\bf 1}_{x<0}+p^{r,-}_Ax{\bf 1}_{x>0}$.  Note that $\hat \psi_A=\partial_x \hat \phi_A$.  For $\ep>0$, let us set
$$
u^\ep(t,x) = \max \{ \hat \phi_\alpha(x)-\alpha(t+\ep) , \hat \phi_\beta(x)-\beta(t+\ep)\}.
$$
Note that $u^\ep$ is a viscosity solution to the HJ equation in $(0,\infty)\times (\R\backslash\{0\})$ as the maximum of two solutions, the Hamiltonian being concave. Hence $\rho^\ep=\partial_x u^\ep$ is an entropic solution of the CL in  $(0,\infty)\times (\R\backslash\{0\})$. On the other hand, as $\alpha <\beta $,  there exists $\eta>0$ such that $u^\ep(t,x)=\hat \phi_\alpha(x)-\alpha(t+\ep)$ in  $\R_+\times (-\eta, \eta)$. Thus $\rho^\ep= \hat \psi_\alpha$ in  $(0, \infty)\times (-\eta, \eta)$. As $ \hat \psi_\alpha$ is a solution for $S$, while $\rho^\ep$ is a solution to the CL in $(0,\infty)\times (\R\backslash\{0\})$,  $\rho^\ep$ is also a solution for $S$ thanks to the finite speed of propagation property. Let us set $\rho:= \partial_x u$  with 
$$
u(t,x) = \max \{ \hat \phi_\alpha(x)-\alpha t , \hat \phi_{\beta}(x)-\beta t\}.
$$
Then $\rho^\ep$ tends to $\rho$ in $L^1_{loc}$, which implies that $\rho$ is also a solution for $S$ by continuity of the semi-group. As $\rho(0)=  \hat \psi_\beta$, we have therefore $S(t, \hat \psi_\beta)= \rho$. On the other hand, $\rho(t)$ equals $\hat \psi_\alpha$ in a neighborhood of $x=0$ for $t>0$: this shows that $ \hat \phi_\beta$ cannot be a stationary solution for $S$ and proves the claim. 

We now come back to the proof of the Lemma. As by assumption $\hat \psi_{\bar A}$ is a stationary solution for $S$, the claim above shows that $\hat \phi_\alpha$ cannot be a stationary solution for $S$ for any $\alpha\in [0,A_{\max}]\backslash\{\bar A\}$. 

Let now $\beta \in (\bar A, A_{\max}]$. Replacing $\hat \psi_\alpha$ by $\hat \psi_{\bar A}$ and $\hat \psi_\beta$ by $\check \psi_\beta$ in the proof of the claim, we can show exactly in the same way as above that $\check \psi_\beta$  cannot be a stationary solution for $S$. 
\end{proof}

We are now ready to build the flux limiter $\bar A$. Let us set $q^\pm=S(1,\hat \psi_{A_{\max}})(0^\pm)$, where $\hat \psi_{A_{\max}}$ is defined in \eqref{def.psipsipsi}. By Lemma  \ref{lem.RH}, $H^l(q^-)=H^r(q^+)$. We set $\bar A= H^l(q^-)=H^r(q^+)$. 

\begin{Lemma} \label{lem.pcannot}  We have $(q^-,q^+)=(p^{l,+}_{\bar A}, p^{r,-}_{\bar A})$ and $\hat \psi_{\bar A}$ is a stationary solution for $S$. Moreover, for any $\alpha\in (\bar A, A_{\max}]$, the maps $x\to p^{l,+}_\alpha{\bf 1}_{x<0}+ p^{l,+}_\alpha{\bf 1}_{x>0}$ and 
$x\to p^{l,-}_\alpha{\bf 1}_{x<0}+ p^{r,-}_\alpha{\bf 1}_{x>0}$ cannot be a stationary solution for $S$.   
\end{Lemma} 

\begin{proof} By Lemma \ref{lem.HJvsCL}, we know that, if we set $\rho= S(\hat \psi_{A_{\max}})$, then there exists $u_0\in Lip$ such that $\rho=\partial_x u$, where $u$ is the  unique viscosity solution to \eqref{eq.HJbis} satisfying the boundary condition \eqref{eq.boundarycond} with  $A$ replaced by $\bar A$. One easily checks, as Step~3 in the proof of Theorem \ref{thm.main1}, that 
$$
u(t,x) = \max \{ u_0(x) -A_{\max} t , \hat \phi_{\bar A}(x)-\bar A t\}.
$$ 
As $u(t,x)= \hat \phi_{\bar A}(x)-\bar A t$ in a neighborhood of $x=0$, we infer that 
$$
(\rho(t,0^-), \rho(t,0^+))= (\partial_x\hat \phi_{\bar A}(0^-), \partial_x\hat \phi_{\bar A}(0^+))=(p^{l,+}_{\bar A},p^{r,-}_{\bar A}).
$$
Moreover Lemma \ref{lem.statsol} states that $\hat \psi_{\bar A}$ is a  stationary solution for $S$. 

Let now $\alpha\in (\bar A, A_{\max}]$. Assume that the map $\tilde \rho(x):=  p^{l,+}_\alpha{\bf 1}_{x<0}+ p^{l,+}_\alpha{\bf 1}_{x>0}$  is a stationary solution for $S$. As $\tilde \rho \geq \hat \psi_{A_{\max}}$, we infer by comparison that $S(t, \tilde \rho)(x)\geq S(t,\hat \psi_{A_{\max}})(x)=\rho(t,x)$ a.e.. Thus $S(1, \tilde \rho)(0^-)\geq p^{l,+}_{\bar A}$. But, as by assumption $\tilde \rho$ is a stationary solution, we have  $p^{l,+}_\alpha = S(1, \tilde \rho)(0^-)\geq p^{l,+}_{\bar A}$, which contradicts the assumption $\alpha>\bar A$. In the same way, one easily checks that the map 
$x\to p^{l,-}_\alpha{\bf 1}_{x<0}+ p^{r,-}_\alpha{\bf 1}_{x>0}$  cannot be a stationary solution for $S$ if $\alpha\in (\bar A, A_{\max}]$.
\end{proof}

\begin{proof}[Proof of Theorem \ref{thm.main2}] Recall that $\bar A= H^l(q^-)=H^r(q^+)$ where $q^\pm=S(1,\hat \psi_0)(0^\pm)$. Let us prove that $S= S^{\bar A}$. Fix $\rho_0\in L^\infty$ and $\rho(t,x)= S(t, \rho_0)(x)$. We know that $\rho$ is an entropy solution to the CL in $\R_+\times (\R\backslash\{0\})$ and that $\rho$ satisfies the initial condition. Let us check that $(\rho(\bar t,0^-),\rho(\bar t,0^+))\in \mathcal G_{\bar A}$ a.e. 

Indeed, for a.e.  $\bar t>0$,  the trace $(q^-,q^+)= (\rho(\bar t,0^-),\rho(\bar t,0^+))$ exists and the map $s:(t,x)\to q^-{\bf 1}_{\{x<0\}}+q^+{\bf 1}_{\{x>0\}}$ is a stationary solution for $S$ (Lemma \ref{lem.statsol}).  
Let us assume for a while that $(q^-,q^+)\notin \mathcal G_{\bar A}$. By Lemma \ref{lem.RH}, we have $H^l(q^-)=H^r(q^+)$. Let $A= H^l(q^-)=H^r(q^+)$.  If $A>\bar A$, then either $s(x)=p^{l,+}_\alpha{\bf 1}_{x<0}+ p^{l,+}_\alpha{\bf 1}_{x>0}$, or $s(x)= p^{l,-}_\alpha{\bf 1}_{x<0}+ p^{r,-}_\alpha{\bf 1}_{x>0}$, or  $s=\hat \psi_A$, or $s= \check \psi_A$. The first two cases are excluded by Lemma \ref{lem.pcannot} while the last two ones are excluded by Lemma \ref{lem.psicannot}. Thus $A\leq \bar A$. As $(q^-,q^+)\notin \mathcal G_{\bar A}$, we necessarily have $A<\bar A$, $q^-=p^{l,+}_A$ and $q^+=p^{r,-}_A$. Thus $s=\hat \psi_A$. But Lemma \ref{lem.psicannot} excludes also this case. Hence $(q^-,q^+)\in \mathcal G_{\bar A}$, which  concludes the proof. 
\end{proof}

%%%%%%%%%%%%%%%%%%%%%%%%%%%%%%%%%%%%%%%%%%%%%%%%
%%%%%%%%%%%%%%%%%%%%%%%%%%%%%%%%%%%%%%%%%%%%%%%%
\paragraph{\textbf{Acknowledgement.}}
This research was partially funded by l'Agence Nationale de la Recherche (ANR), project ANR-22-CE40-0010 COSS. 
For the purpose of open access, the authors have applied a CC-BY public copyright licence to any Author Accepted Manuscript (AAM) version arising from this submission. The author thanks Boris Andreianov and R\'egis Monneau for fruitful discussions.

\end{document}